\documentclass[12pt,reqno]{article}

\usepackage[usenames]{color}
\usepackage{amssymb}
\usepackage{amsmath}
\usepackage{amsthm}
\usepackage{amsfonts}
\usepackage{amscd}
\usepackage{graphicx}

\usepackage[colorlinks=true,
linkcolor=webgreen,
filecolor=webbrown,
citecolor=webgreen]{hyperref}

\definecolor{webgreen}{rgb}{0,.5,0}
\definecolor{webbrown}{rgb}{.6,0,0}

\usepackage{color}
\usepackage{fullpage}
\usepackage{float}

\usepackage{graphics}
\usepackage{latexsym}

\DeclareMathOperator{\arctanh}{arctanh}
\DeclareMathOperator{\Li}{Li}
\DeclareMathOperator{\Cl}{Cl}
\DeclareMathOperator{\Gl}{Gl}

\begin{document}

\theoremstyle{plain}
\newtheorem{theorem}{Theorem}
\newtheorem{corollary}[theorem]{Corollary}
\newtheorem{lemma}{Lemma}
\newtheorem*{example}{Examples}
\newtheorem*{remark}{Remark}

\begin{center}
\vskip 1cm{\LARGE\bf 
A series of Ramanujan, two-term dilogarithm identities and some Lucas series \\
}
\vskip 1cm
{\large

Kunle Adegoke \\
Department of Physics and Engineering Physics \\ Obafemi Awolowo University, Ile-Ife \\ Nigeria \\
\href{mailto:adegoke00@gmail.com}{\tt adegoke00@gmail.com}

\vskip 1cm

Robert Frontczak\footnote{Statements and conclusions made in this article by R. F. are entirely those of the author. 
They do not necessarily reflect the views of LBBW.} \\
Landesbank Baden-W\"urttemberg (LBBW), Stuttgart \\ Germany \\
\href{mailto:robert.frontczak@lbbw.de}{\tt robert.frontczak@lbbw.de}
}

\end{center}

\vskip .2 in

\begin{abstract}
We study an elementary series that can be considered a relative of a series studied by Ramanujan in Part 1 of his Lost Notebooks. 
We derive a closed form for this series in terms of the inverse hyperbolic arctangent and the polylogarithm. 
Special cases will follow in terms of the Riemann zeta and the alternating Riemann zeta function. 
In addition, some trigonometric series will be expressed in terms of the Clausen functions.
Finally, a range of new two-term dilogarithm identities will be proved and some difficult series involving Lucas numbers 
will be evaluated in closed form.  
\end{abstract}



\bigskip

\section{Introduction}

In Part 1 of Ramanujan's Lost Notebooks \cite{Berndt1} Ramanujan studied the function
\begin{equation*}
\varphi(a,n) = 1 + 2 \sum_{k=1}^n \frac{1}{(ak)^3-ak}, \qquad \varphi(a) = \lim_{n\rightarrow\infty}\varphi(a,n).
\end{equation*}
Among other things Ramanujan proved the identities
\begin{equation}\label{Rama1}
\ln (2) = \frac{1}{2} \varphi (2) = \frac{1}{2} + \sum_{k=1}^\infty \frac{1}{(2k-1)(2k)(2k+1)}
\end{equation}
and
\begin{equation}
\frac{3}{2} \ln (2) = \varphi (4) = 1 + \sum_{k=1}^\infty \frac{2}{(4k-1)(4k)(4k+1)}.
\end{equation}
Ramanujan also considered the alternating variant of $\varphi(a)$, i.e., the function
\begin{equation*}
\widetilde{\varphi}(a) = 1 + 2 \sum_{k=1}^\infty \frac{(-1)^k}{(ak)^3-ak}.
\end{equation*}
He obtained \cite[p. 40/41]{Berndt1}
\begin{equation}\label{Rama2}
\ln (2) = \widetilde{\varphi}(2) \qquad\mbox{and}\qquad \frac{4}{3} \ln (2) = \widetilde{\varphi}(3).
\end{equation}

The function $\varphi(a,n)$ has also been studied recently. It appears in an article by Berndt and Huber from 2008 \cite{Berndt2} 
who derive a new formula for the Euler-Maccheroni constant $\gamma$. In addition, we mention the article by Ravichandran 
from 2004 \cite{Ravi} where the function $\varphi(a)$ has been analyzed (under the notation $A_n$). \\

In this article, we study a relative of Ramanujan's function $\varphi(a)$, namely the series
\begin{equation}
F(z,p) = \sum_{k=1}^\infty \frac{z^{2k}}{(2k-1)(2k)^p (2k+1)}, \qquad z\in\mathbb{C}, p\geq 0.
\end{equation}
We express the series in closed form using the inverse hyperbolic arctangent, $\arctanh (z)$, 
and the polylogarithm of order $n$, $\Li_n (z)$, the later being defined by \cite{Lewin}
\begin{equation*}
\Li_n (z) = \sum_{k=1}^\infty \frac{z^k}{k^n}, \qquad |z|<1. 
\end{equation*}
At $z=1$ we evaluate the series in terms of the Riemann zeta function given by
\begin{equation*}
\zeta (s) = \sum_{k=1}^{\infty} \frac{1}{k^{s}}, \qquad \Re(s) >1.
\end{equation*}
We also show that at $z=i,i=\sqrt{-1},$ $F(z,p)$ can be expressed in terms of alternating zeta, 
or Dirichlet's eta, or Euler's eta function, $\eta(s)$, defined by
\begin{equation*}
\eta(s) = \sum_{k=1}^\infty \frac{(-1)^{k+1}}{k^s}, \quad \Re(s)>0.
\end{equation*}
Other difficult series will follow as particular cases. One such series is
\begin{equation*}
\sum_{k=1}^\infty \frac{1}{2^{k+3}(2k-1) k^3 (2k+1)} = \frac{3}{2\sqrt{2}} \ln(1+\sqrt{2}) + \frac{\pi^2 \ln(2)}{96} 
- \frac{1}{2} - \frac{\ln(2)}{2} - \frac{\ln^3 (2)}{48} - \frac{7}{64} \zeta(3).
\end{equation*}
In addition, as corollaries to our main theorem, we will express some trigonometric series in terms of the Clausen functions.
Finally, we prove a range of new two-term dilogarithm identities and evaluate some difficult 
series involving Lucas numbers $L_n.$

\section{The main result}

Our approach is completely elementary and builds mainly on properties of the inverse hyperbolic arctangent $\arctanh (z)$.
For complex arguments $z$, $\arctanh (z)$ is defined by \cite{Abram}
\begin{equation*}
\arctanh (z) = \int_0^z \frac{dt}{1-t^2}.
\end{equation*}
It is a multivalued function with a branch cut in the complex plane. It is an odd function related to the inverse tangent via
\begin{equation*}
\arctanh (z) = - i \arctan (iz), \qquad i=\sqrt{-1}.
\end{equation*}
The logarithmic representation of the inverse hyperbolic arctangent is given by
\begin{equation}
\arctanh (z) = \frac{1}{2}\ln \Big ( \frac{1+z}{1-z}\Big ), \qquad \vert z\vert < 1.
\end{equation}
It is also well known that it possess a Maclaurin series expansion of the form
\begin{equation}\label{macl}
\arctanh (z) = \sum_{k=1}^\infty \frac{z^{2k-1}}{2k-1}, \qquad \vert z\vert < 1.
\end{equation}
Recent articles on the function include \cite{Chen,Frontczak,Zhu}. 

To prove our main result we will need the following lemmas.

\begin{lemma}\label{lem1}
We have the following limits:
\begin{equation*}
\lim_{z \rightarrow 0} \left (z + \frac{1}{z}\right ) \arctanh (z) = 1 \qquad \mbox{and}\qquad 
\lim_{z\rightarrow 0} \left (z - \frac{1}{z}\right ) \arctanh (z) = -1.
\end{equation*}
\end{lemma}
\begin{proof}
Both limits are proved easily using the rule of l'Hospital.
\end{proof}

\begin{lemma}\label{lem2}
The following expressions are valid:
\begin{equation*}
\int \frac{1}{z} \left (1 + \left (z - \frac{1}{z}\right ) \arctanh (z)\right ) dz 
= \ln (1-z^2) + \left (z + \frac{1}{z}\right ) \arctanh (z) + c,
\end{equation*}
\begin{equation*}
\int \frac{1}{z} \left (-1 + \left (z + \frac{1}{z}\right ) \arctanh (z)\right ) dz = \left (z - \frac{1}{z}\right ) \arctanh (z) + c,
\end{equation*}
\begin{equation*}
\int \frac{ \ln (1-z^2)}{z} dz = - \frac{1}{2} \Li_2(z^2) + c
\end{equation*}
and
\begin{equation*}
\int \frac{\Li_n(z^2)}{z} dz = \frac{1}{2} \Li_{n+1}(z^2) + c.
\end{equation*}
\end{lemma}
\begin{proof}
The integrals are derived by standard methods or verified directly by differentiation. 
\end{proof}

\begin{theorem}\label{thm_1}
For any integer $p\geq 0$ and $|z|< 1$ we have
\begin{equation}\label{main_1}
F(z,p) = \begin{cases}
	\frac{1}{2} \left ( \left (z - \frac{1}{z}\right ) \arctanh (z) + 1 - \sum_{j=1}^{p/2} 2^{-(2j-1)} \Li_{2j}(z^2) \right ), 
	& \text{\rm $p$ even;} \\[16pt]
	\frac{1}{2} \left ( \left (z + \frac{1}{z}\right ) \arctanh (z) - 1 + \ln(1-z^2) - \sum_{j=1}^{(p-1)/2} 2^{-2j} \Li_{2j+1}(z^2) \right ), & \text{\rm $p$ odd.} 
\end{cases} 
\end{equation}
\end{theorem}
\begin{proof}
We prove the identities using induction on the parameter $p$. We start with the base cases $p=0$ and $p=1$, respectively. 
Employing \eqref{macl} we get
\begin{eqnarray*}
\Big (z - \frac{1}{z}\Big ) \arctanh (z) & = & \sum_{k=1}^\infty \frac{z^{2k}}{2k-1} - \sum_{k=1}^\infty \frac{z^{2k-2}}{2k-1} \\
& = & \sum_{k=1}^\infty \frac{z^{2k}}{2k-1} - 1 - \sum_{k=1}^\infty \frac{z^{2k}}{2k+1} \\
& = & -1 + 2 \sum_{k=1}^\infty \frac{z^{2k}}{(2k-1)(2k+1)}.
\end{eqnarray*}
This proves the base case $p=0$. Next, we have
\begin{equation*}
\sum_{k=1}^\infty \frac{z^{2k}}{(2k-1)(2k)(2k+1)} = \sum_{k=1}^\infty \int_0^z \frac{t^{2k-1}\, dt}{(2k-1)(2k+1)} 
= \int_0^z \frac{1}{2t} \Big (1 + \Big (t - \frac{1}{t}\Big ) \arctanh (t) \Big ) dt. 
\end{equation*}
Using the first integral in Lemma \ref{lem2} we get
\begin{equation*}
\sum_{k=1}^\infty \frac{z^{2k}}{(2k-1)(2k)(2k+1)} = \frac{1}{2} \Big ( \Big (z + \frac{1}{z}\Big ) \arctanh (z) - 1 + \ln(1-z^2) \Big )
\end{equation*}
as desired. Now let $p\geq 2$ be arbitrary and even. Then $p+1$ is odd and using Lemmas \ref{lem1} and \ref{lem2}
\begin{eqnarray*}
\sum_{k=1}^\infty \frac{z^{2k}}{(2k-1)(2k)^{p+1} (2k+1)} & = & 
\int_0^z \frac{1}{2t} \Big (1 + \Big (t - \frac{1}{t}\Big ) \arctanh (t) - \sum_{j=1}^{p/2} \frac{1}{2^{2j-1}} \Li_{2j}(t^2) \Big ) dt \\
& = & \frac{1}{2} \Big (\Big (z + \frac{1}{z}\Big ) \arctanh (z) - 1 + \ln(1-z^2) - \sum_{j=1}^{p/2} \frac{1}{2^{2j}} \Li_{2j+1}(z^2) \Big )
\end{eqnarray*}
and the statement is proved. Similarly, if $p$ is odd then $p+1$ is even and
\begin{eqnarray*}
\sum_{k=1}^\infty \frac{z^{2k}}{(2k-1)(2k)^{p+1} (2k+1)} & = & 
\int_0^z \frac{1}{2t} \Big (\Big (t + \frac{1}{t}\Big ) \arctanh (t) - 1 + \ln(1-t^2) \\
& & \qquad - \sum_{j=1}^{(p-1)/2} \frac{1}{2^{2j}} \Li_{2j+1}(t^2) \Big ) dt \\
& = & \frac{1}{2} \Big (\Big (z - \frac{1}{z}\Big ) \arctanh (z)  + 1 - \frac{1}{2}\Li_2(z^2) 
- \sum_{j=1}^{(p-1)/2} \frac{1}{2^{2j+1}} \Li_{2j+2}(z^2) \Big ) \\
& = & \frac{1}{2} \Big (\Big (z - \frac{1}{z}\Big ) \arctanh (z)  + 1 - \sum_{j=1}^{(p+1)/2} \frac{1}{2^{2j-1}} \Li_{2j}(z^2) \Big ). 
\end{eqnarray*}
\end{proof}

In particular, for $|z|<1$, we have
\begin{equation}\label{eq.ajavxbw}
F(z,1)=\sum_{k = 1}^\infty  {\frac{{z^{2k} }}{{(2k - 1)(2k)(2k + 1)}}}  = \frac{1}{2}\ln (1 - z^2 ) + \frac{1}{2}\left( {z + \frac{1}{z}} \right)\arctanh(z)-\frac12,
\end{equation}
\begin{equation}\label{eq.ekwvw5d}
F(z,2)=\sum_{k = 1}^\infty \frac{{z^{2k} }}{{(2k - 1)(2k)^2 (2k + 1)}} = -\frac14 \Li_2 (z^2) - \frac12\left( {\frac{{1 - z^2}}{z^2}} \right) z \arctanh (z)
 + \frac12.
\end{equation}
and
\begin{equation}\label{eq.uu2c0rv}
\begin{split}
F(z,3)=\sum_{k = 1}^\infty \frac{{z^{2k} }}{{(2k - 1)(2k)^3(2k + 1)}}  &= - \frac{1}{4} \frac{(1 - z)^2 }{z} \ln (1 - z) 
+ \frac{1}{4}\frac{{(1 + z)^2 }}{z}\ln (1 + z) \\
&\qquad - \frac{1}{8} \Li_3 (z^2 ) - \frac{1}{2}.
\end{split}
\end{equation}

\begin{corollary}\label{cor1}
We have for $p\geq 0$ 
\begin{eqnarray}\label{ex1}
F(1,p) & = & \sum_{k=1}^\infty \frac{1}{(2k-1)(2k)^{p} (2k+1)} \nonumber \\
& = & \begin{cases}
	\frac{1}{2} - \sum_{j=1}^{p/2} 2^{-2j} \zeta(2j), & \text{\rm $p$ even;} \\[16pt]
	-\frac{1}{2} + \ln(2) - \sum_{j=1}^{(p-1)/2} 2^{-(2j+1)} \zeta(2j+1), & \text{\rm $p$ odd.} 
\end{cases} 
\end{eqnarray}
\end{corollary}
\begin{proof}
Take the limit $z\rightarrow 1$ and make use of the results $\Li_n(1) = \zeta(n)$ as well as 
\begin{equation*}
\lim_{z\rightarrow 1} \left (z - \frac{1}{z}\right ) \arctanh (z) = 0 \quad \mbox{and}\quad 
\lim_{z\rightarrow 1} \left ( \left (z + \frac{1}{z}\right ) \arctanh (z) + \ln(1-z^2) \right ) = 2 \ln(2).
\end{equation*}
\end{proof}

\begin{corollary}\label{cor2}
We have for $p\geq 0$ 
\begin{eqnarray}\label{ex2}
F(i,p) & = & \sum_{k=1}^\infty \frac{(-1)^k}{(2k-1)(2k)^{p} (2k+1)} \nonumber \\
& = & \begin{cases}
	\frac{1}{2} - \frac{\pi}{4} + \sum_{j=1}^{p/2} 2^{-2j} \eta (2j), & \text{\rm $p$ even;} \\[16pt]
	- \frac{1}{2} + \frac{\ln(2)}{2} + \sum_{j=1}^{(p-1)/2} 2^{-(2j+1)} \eta (2j+1), & \text{\rm $p$ odd,} 
\end{cases} 
\end{eqnarray}
where $\eta(s)$ denotes the Dirichlet eta function.
\end{corollary}
\begin{proof}
Take the limit $z\rightarrow i$ and make use of the results $\Li_n(-1) = -\eta(n)$ \cite{Lewin} as well as 
\begin{equation*}
\lim_{z\rightarrow i} \left (z - \frac{1}{z}\right ) \arctanh (z) = - \frac{\pi}{2} \quad \mbox{and}\quad 
\lim_{z\rightarrow i} \left ( \left (z + \frac{1}{z}\right ) \arctanh (z) + \ln(1-z^2) \right ) = \ln(2).
\end{equation*}
\end{proof}
\begin{remark}
The relation
\begin{equation}\label{eta_zeta}
\eta(s) = (1-2^{1-s})\zeta(s)
\end{equation}
gives
\begin{eqnarray}\label{ex3}
F(i,p) & = & \sum_{k=1}^\infty \frac{(-1)^k}{(2k-1)(2k)^{p} (2k+1)} \nonumber \\
& = & \begin{cases}
	\frac{1}{2} - \frac{\pi}{4} + \sum_{j=1}^{p/2} (2^{-2j} - 2^{-(4j-1)}) \zeta (2j), & \text{\rm $p$ even;} \\[16pt]
	- \frac{1}{2} + \frac{\ln(2)}{2} + \sum_{j=1}^{(p-1)/2} (2^{-(2j+1)} - 2^{-(4j+1)}) \zeta (2j+1), & \text{\rm $p$ odd.} 
\end{cases} 
\end{eqnarray}
\end{remark}

Some particular cases of the previous results are stated below:

\begin{equation*}
\sum_{k=1}^\infty \frac{1}{(2k-1)(2k)(2k+1)} = \ln(2) - \frac{1}{2},
\end{equation*}
\begin{equation*}
\sum_{k=1}^\infty \frac{(-1)^k}{(2k-1)(2k)(2k+1)} = \frac{1}{2}(\ln(2) - 1),
\end{equation*}
\begin{equation*}
\sum_{k=1}^\infty \frac{1}{(2k-1)(2k)^2(2k+1)} = \frac{1}{2} - \frac{\pi^2}{24},
\end{equation*}
\begin{equation*}
\sum_{k=1}^\infty \frac{(-1)^k}{(2k-1)(2k)^2(2k+1)} = \frac{1}{2} - \frac{\pi}{4} + \frac{\pi^2}{48},
\end{equation*}
\begin{equation*}
\sum_{k=1}^\infty \frac{1}{(2k-1)(2k)^3(2k+1)} = \ln(2) - \frac{1}{2} - \frac{1}{8}\zeta(3),
\end{equation*}
\begin{equation*}
\sum_{k=1}^\infty \frac{(-1)^k}{(2k-1)(2k)^3(2k+1)} = \frac{1}{2}(\ln(2) - 1) + \frac{3}{32} \zeta(3).
\end{equation*}

The first identity is a rediscovery of Ramanujan's identity \eqref{Rama1} while the second recovers his first result in \eqref{Rama2}.

\begin{corollary}\label{cor3}
We have 
\begin{equation}
\sum_{k=1}^\infty \frac{1}{2^{k+1}(2k-1) k (2k+1)} = \frac{3}{2\sqrt{2}} \ln(1+\sqrt{2}) - \frac{1}{2}(\ln(2)+1),
\end{equation}
\begin{equation}
\sum_{k=1}^\infty \frac{1}{2^{k+2}(2k-1) k^2 (2k+1)} = \frac{1}{2} + \frac{1}{8} \ln^2 (2) - \frac{\pi^2}{48} 
- \frac{1}{2\sqrt{2}}\ln(1+\sqrt{2}),
\end{equation}
\begin{equation}
\sum_{k=1}^\infty \frac{1}{2^{k+3}(2k-1) k^3 (2k+1)} = \frac{3}{2\sqrt{2}} \ln(1+\sqrt{2}) + \frac{\pi^2 \ln(2)}{96}
- \frac{1}{2}(\ln(2)+1) - \frac{\ln^3 (2)}{48} - \frac{7}{64}\zeta(3).
\end{equation}
\end{corollary}
\begin{proof}
Set $z=1/\sqrt{2}$ in \eqref{main_1}, use
\begin{equation*}
\Li_2\left (\frac{1}{2}\right ) = \frac{1}{2}\left ( \zeta (2) - \ln^2 (2)\right )
\end{equation*}
and
\begin{equation*}
\Li_3\left (\frac{1}{2}\right ) = \frac{7}{8}\zeta(3) - \frac{1}{2} \ln(2) \zeta(2) + \frac{1}{6} \ln^3 (2),
\end{equation*}
and simplify.
\end{proof}

\begin{corollary}\label{cor.duqhfnh}
We have
\begin{gather}
\sum_{k = 1}^\infty \frac{1}{{5^k (2k - 1)(2k)(2k + 1)}} 
= \frac{1}{2}\ln \left( {\frac{4}{5}} \right) + \frac{3}{\sqrt{5}}\ln (\alpha) - \frac{1}{2}, \\
\sum_{k = 1}^\infty \frac{{4^k }}{{5^k (2k - 1)(2k)(2k + 1)}} =  - \frac{1}{2}\ln (5) + \frac{27}{4\sqrt{5}} \ln (\alpha) - \frac{1}{2}, \\
\sum_{k = 1}^\infty \frac{{5^k }}{{9^k (2k - 1)(2k)(2k + 1)}} 
= \ln \left( {\frac{2}{3}} \right) + \frac{14}{3\sqrt{5}}\ln (\alpha) - \frac{1}{2} \label{eq.wjq5pc1}.
\end{gather}
\end{corollary}
\begin{proof}
Evaluate~\eqref{eq.ajavxbw} at $z=1/\sqrt 5$, $z=2/\sqrt 5$ and $z=\sqrt 5/3$.
\end{proof}

\begin{corollary}\label{cor.j78scl1}
If $p$ is a non-negative integer and $0<x<2\pi$, then
\begin{equation}\label{eq.uygdx2h}
\begin{split}
&\sum_{k = 1}^\infty  {\frac{{\cos (kx)}}{{(2k - 1)(2k)^p (2k + 1)}}}\\
&\qquad =  \begin{cases}
 \frac{1}{2} - \frac{\pi }{4}\sin \left( {\frac{x}{2}} \right) - \sum\limits_{j = 1}^{p/2} {\frac{{\Gl_{2j} (x)}}{{2^{2j} }}},\quad\mbox{$p$ even} ;  \\ 
  - \frac{1}{2} - \frac{1}{2}\cos \left( {\frac{x}{2}} \right)\ln \left( {\tan \left( {\frac{x}{4}} \right)} \right) + \frac{1}{2}\ln \left( {2\sin \left( {\frac{x}{2}} \right)} \right) - \sum\limits_{j = 1}^{(p - 1)/2} {\frac{{\Cl_{2j + 1} (x)}}{{2^{2j + 1} }}},\quad\mbox{$p$ odd};  \\ 
 \end{cases} 
\end{split}
\end{equation}
and
\begin{equation}
\begin{split}
&\sum_{k = 1}^\infty  {\frac{{\sin (kx)}}{{(2k - 1)(2k)^p (2k + 1)}}} \\
&\qquad =  \begin{cases}
  - \frac{1}{2}\sin \left( {\frac{x}{2}} \right)\ln \left( {\tan \left( {\frac{x}{4}} \right)} \right) - \sum\limits_{j = 1}^{p/2} {\frac{{\Cl_{2j} (x)}}{{2^{2j} }}},\quad\mbox{$p$ even} ;  \\ 
 \frac{\pi }{4}\cos \left( {\frac{x}{2}} \right) - \frac{1}{4}(\pi  - x) - \sum\limits_{j = 1}^{(p - 1)/2} {\frac{{\Gl_{2j + 1} (x)}}{{2^{2j + 1} }}},\quad\mbox{$p$ odd} ;  \\ 
 \end{cases} 
\end{split}
\end{equation}
where $\Cl_n(y)$ and $\Gl_n(y)$ are the Clausen functions~\cite{Lewin} defined through
\begin{gather*}
\Li_{2n} (e^{ix} ) = \Gl_{2n} (x) + i\Cl_{2n} (x),\\
\Li_{2n + 1} (e^{ix} ) = \Cl_{2n + 1} (x) + i\Gl_{2n + 1} (x);
\end{gather*}
so that
\begin{gather*}
\Cl_{2n} (x) = \sum_{k = 1}^\infty  {\frac{{\sin (kx)}}{{k^{2n} }}} ,\quad\Cl_{2n + 1} (x) = \sum_{k = 1}^\infty  {\frac{{\cos (kx)}}{{k^{2n + 1} }}}, \\
\Gl_{2n} (x) = \sum_{k = 1}^\infty  {\frac{{\cos (kx)}}{{k^{2n} }}} ,\quad\Gl_{2n + 1} (x) = \sum_{k = 1}^\infty  {\frac{{\sin (kx)}}{{k^{2n + 1} }}}. 
\end{gather*}
\end{corollary}
\begin{proof}
Set $z=\exp(ix/2)$ in~Theorem~\ref{thm_1} and take real and imaginary parts.
\end{proof}

\begin{corollary}\label{cor.nhgff23}
If $p$ is a non-negative integer, then
\begin{gather}
\sum_{k = 1}^\infty  {\frac{{\cos (kx)}}{{(2k - 1)(2k)^{2p} (2k + 1)}}}  = \frac{1}{2} - \frac{\pi }{4}\sin \left( {\frac{x}{2}} \right) - \frac{1}{2}\sum_{j = 1}^p {\frac{{( - 1)^{j + 1} \pi ^{2j} }}{{(2j)!}}B_{2j} \left( {\frac{x}{{2\pi }}} \right)}, \\
\sum_{k = 1}^\infty  {\frac{{\sin (kx)}}{{(2k - 1)(2k)^{2p + 1} (2k + 1)}}}  = \frac{\pi }{4}\cos \left( {\frac{x}{2}} \right) - \frac{1}{4}(\pi  - x) - \frac{1}{2}\sum_{j = 1}^p {\frac{{( - 1)^{j + 1} \pi ^{2j + 1} }}{{(2j + 1)!}}B_{2j + 1} \left( {\frac{x}{{2\pi }}} \right)}\label{eq.pk7z8s2} ,
\end{gather}
where $B_n(y)$ are the Bernoulli polynomials defined by
\[
\frac{{te^{xt} }}{{e^t  - 1}} = \sum_{k = 0}^\infty  {\frac{{B_k (x)t^k }}{{k!}}}.
\]
\end{corollary}
\begin{proof}
Use the identity~\cite{Lewin}
\[
\Gl_n (x) = ( - 1)^{1 + \left\lfloor {n/2} \right\rfloor } 2^{n - 1} \pi ^n \frac{{B_n (x/2\pi )}}{{n!}}
\]
in Corollary~\ref{cor.j78scl1}.
\end{proof}

\begin{theorem}
If $p$ is a non-negative integer, then
\begin{equation}\label{eq.jt6n9me}
\begin{split}
\sum_{k = 1}^\infty \frac{(- 1)^{k - 1}}{(4k - 1)(4k)^{2p + 1} (4k + 1)}  &= \frac{1}{2} + \frac{{\sqrt 2 }}{4}\ln (\sqrt 2  - 1) 
- \frac{1}{4}\ln (2) - \frac{1}{4}\sum_{j = 1}^p \frac{{(1 - 2^{- 2j} )}}{{2^{4j} }}\zeta (2j + 1),
\end{split}
\end{equation}
\begin{equation}\label{eq.fb5dwvj}
\sum_{k = 1}^\infty \frac{(- 1)^{k - 1} }{(4k - 3)(4k - 2)^{2p + 1} (4k - 1)} = \frac{\pi }{8}(\sqrt 2  - 1) 
- \frac{1}{2}\sum_{j = 1}^p \frac{{( - 1)^{j + 1} \pi ^{2j + 1} }}{{(2j + 1)!}}B_{2j + 1} (1/4).
\end{equation}
\end{theorem}
\begin{proof}
Write $2p + 1$ for $p$ in~\eqref{eq.uygdx2h}, set $x=\pi/2$ and use the identity~\cite{Lewin}
\[
\Cl_{2n + 1} (\pi /2) = - 2^{-(2n + 1)} (1 - 2^{-2n}) \zeta (2n + 1)
\]
to obtain
\begin{equation*}
\sum_{k = 1}^\infty \frac{\cos (k\pi /2)}{(2k - 1)(2k)^{2p + 1} (2k + 1)} = - \frac{1}{2} - \frac{{\sqrt 2 }}{4}\ln (\sqrt 2  - 1) 
+ \frac{1}{4}\ln (2) + \frac{1}{4}\sum_{j = 1}^p \frac{{(1 - 2^{ - 2j} )}}{{2^{4j} }}\zeta (2j + 1),
\end{equation*}
from which~\eqref{eq.jt6n9me} follows. Identity~\eqref{eq.fb5dwvj} is proved by setting $x=\pi/2$ in~\eqref{eq.pk7z8s2}.
\end{proof}

\begin{theorem}
If $p$ is a non-negative integer, then
\begin{equation}\label{eq_last1}
\begin{split}
&\sum_{k = 1}^\infty \frac{1}{(6k - 1)(6k)^{2p + 1} (6k + 1)} 
- \frac{1}{2} \Big ( \sum_{k = 1}^\infty \frac{1}{(6k + 1)(6k+2)^{2p + 1} (6k + 3)} \\
& \qquad + \sum_{k = 1}^\infty \frac{1}{(6k + 3)(6k+4)^{2p + 1} (6k + 5)}\Big )  \\ 
& \qquad = \frac{1}{3}\frac{1}{ 2^{2p+2}} + \frac{1}{15}\frac{1}{ 2^{4p+3}} - \frac{1}{2} + \frac{3}{8}\ln (3) 
+ \frac{1}{4}\sum_{j = 1}^p \frac{(1-3^{-2j})}{2^{2j}}\zeta (2j + 1).
\end{split}
\end{equation}
and
\begin{equation}\label{eq_last2}
\begin{split}
&\sum_{k = 0}^\infty \frac{1}{(6k + 1)(6k+2)^{2p + 1} (6k + 3)} - \sum_{k = 0}^\infty \frac{1}{(6k + 3)(6k+4)^{2p + 1} (6k + 5)} \\ 
& \qquad = \frac{\sqrt{3} \pi}{36} - \frac{\sqrt{3}}{3} \sum_{j = 1}^p \frac{( - 1)^{j + 1} \pi^{2j + 1}}{(2j + 1)!} B_{2j + 1} (1/3).
\end{split}
\end{equation}
\end{theorem}
\begin{proof}
Again, write $2p + 1$ for $p$ in \eqref{eq.uygdx2h}, set $x=2\pi/3$ and use the identity \cite{Lewin}
\[
\Cl_{2n + 1} (2 \pi /3) = - \frac{1}{2} (1 - 3^{-2n}) \zeta (2n + 1)
\]
to obtain
\begin{equation*}
\sum_{k = 1}^\infty \frac{\cos (2\pi k/3)}{(2k - 1)(2k)^{2p + 1} (2k + 1)} = - \frac{1}{2} + \frac{3}{8}\ln (3) 
+ \frac{1}{4}\sum_{j = 1}^p \frac{(1-3^{-2j})}{2^{2j}}\zeta (2j + 1),
\end{equation*}
from which \eqref{eq_last1} follows upon noting that
\begin{equation*}
\cos (2\pi k/3) = \begin{cases}
1, & \mbox{$k\equiv 0$ mod 3};  \\ 
-1/2, & \mbox{$k\equiv 1,2$ mod 3} ;  \\ 
\end{cases} 
\end{equation*}
and simplifying. Identity \eqref{eq_last2} is proved by setting $x=2\pi/3$ in \eqref{eq.pk7z8s2} in conjunction with
\begin{equation*}
\sin (2\pi k/3) = \begin{cases}
0, & \mbox{$k\equiv 0$ mod 3};  \\ 
\sqrt{3}/2, & \mbox{$k\equiv 1$ mod 3};  \\ 
-\sqrt{3}/2, & \mbox{$k\equiv 2$ mod 3}.  \\
\end{cases} 
\end{equation*}
\end{proof}

Additional interesting results can be obtained from Corollaries \ref{cor.j78scl1} and \ref{cor.nhgff23}, which we leave for a private study.

\section{Another approach to evaluate $F(1,p)$ and $F(i,p)$ and a restatement of Theorem 1} 

There is another direct approach to evaluate the series $F(1,p)$ and $F(i,p)$, respectively.
As $F(1,0)$ and $F(i,0)$ follow easily be telescoping we assume that $p\geq 1$. From the partial fraction decomposition
\begin{equation*}
\frac{1}{(2k-1) k^p (2k+1)} = - \frac{1}{k^p} + \frac{1}{k^{p-1}(2k+1)} + \frac{1}{k^{p-1}(2k-1)}
\end{equation*}
in conjunction with
\begin{equation*}
\frac{1}{k^p (2k+1)} = \frac{1}{k^p} - \frac{2}{k^{p-1}(2k+1)}, \qquad \frac{1}{k^p (2k-1)} = - \frac{1}{k^p} + \frac{2}{k^{p-1}(2k-1)},
\end{equation*}
we get for $p\geq 1$
\begin{equation}\label{pafrac}
\frac{1}{(2k-1) k^p (2k+1)} = - \frac{1}{k^p} + \sum_{j=0}^{p-2} \frac{2^j ((-1)^j - 1)}{k^{p-1-j}} 
+ 2^{p-1}\Big ( \frac{(-1)^{p-1}}{2k+1} + \frac{1}{2k-1} \Big ).
\end{equation}
This produces
\begin{equation*}
F(1,p) = - 2^{-p} \zeta(p) +  \sum_{j=0}^{p-2} \frac{(-1)^j - 1}{2^{p-j}} \zeta(p-1-j) 
+ \frac{1}{2} \sum_{k=1}^\infty \Big ( \frac{1}{2k-1} + \frac{(-1)^{p-1}}{2k+1} \Big ).
\end{equation*}
If $p$ is even, then by telescoping
\begin{equation*}
\sum_{k=1}^\infty \Big ( \frac{1}{2k-1} - \frac{1}{2k+1} \Big ) = 1
\end{equation*}
and
\begin{eqnarray*}
F(1,p) & = & \frac{1}{2} - 2^{-p} \zeta(p) - \sum_{j=1, \, j\, \mbox{\tiny odd}}^{p-2} \frac{1}{2^{p-(j+1)}} \zeta(p-(j+1)) \\
& = & \frac{1}{2} - 2^{-p} \zeta(p) - \sum_{j=2, \, j\, \mbox{\tiny even}}^{p-2} \frac{1}{2^{p-j}} \zeta(p-j) \\
& = & \frac{1}{2} - 2^{-p} \zeta(p) - \sum_{j=1}^{(p-2)/2} \frac{1}{2^{p-2j}} \zeta(p-2j). 
\end{eqnarray*}
This is the identity for $p$ even as stated in \eqref{ex1}. Similarly if $p$ is odd, then 
\begin{equation*}
\sum_{k=1}^\infty \Big ( \frac{1}{2k-1} + \frac{1}{2k+1} \Big ) = 1 + 2 \sum_{k=1}^\infty \frac{1}{2k+1}
\end{equation*}
and
\begin{equation*}
F(1,p) = - \frac{1}{2} - 2^{-p} \zeta(p) + \sum_{k=0}^\infty \frac{1}{2k+1} - \frac{1}{2}\zeta(1) 
+ \sum_{j=0}^{p-3} \frac{(-1)^j - 1}{2^{p-j}} \zeta(p-1-j). 
\end{equation*}
But
\begin{equation*}
\sum_{k=0}^\infty \frac{1}{2k+1} - \frac{1}{2}\zeta(1) = 1 - \frac{1}{2} \sum_{k=1}^\infty \frac{1}{k(2k+1)} 
= 1 - \frac{1}{2}( 2 - 2 \ln (2)) = \ln (2),
\end{equation*}
and hence
\begin{equation*}
F(1,p) = - \frac{1}{2} + \ln (2) - 2^{-p} \zeta(p) - \sum_{j=1}^{(p-3)/2} 2^{-(p-2j)} \zeta(p-2j). 
\end{equation*}
The last expression is equivalent to the second part in \eqref{ex1}. 
The derivation of $F(i,p)$ based on a partial fraction decomposition is done in the same way making use of the results
\begin{equation*}
\sum_{k=1}^\infty \frac{(-1)^k}{(2k-1)(2k+1)} = \frac{2-\pi}{4}
\end{equation*}
and
\begin{equation*}
\sum_{k=1}^\infty \frac{(-1)^k k}{(2k-1)(2k+1)} = - \frac{1}{4}.
\end{equation*}
Both results are proved easily. The first one, for instance, follows from
\begin{equation*}
\frac{1}{(2k-1)(2k+1)} = \frac{1}{2} \Big ( \frac{1}{2k-1} - \frac{1}{2k+1} \Big )
\end{equation*}
which gives after shifting the summation index
\begin{equation*}
\sum_{k=1}^\infty \frac{(-1)^k }{(2k-1)(2k+1)} = \frac{1}{2} - \sum_{k=1}^\infty \frac{(-1)^k}{2k+1}.
\end{equation*}
The other series is treated similarly. The remaining steps are omitted. \\

Surprisingly, the above analysis allows to prove Theorem \ref{thm_1} directly,
without requiring to distinguish between $p$ odd and $p$ even, respectively. 
Note that \eqref{pafrac} is equivalent to
\begin{equation*}
\frac{1}{(2k - 1)k^p (2k + 1)} =  - \frac 1{k^p} + \sum_{j = 1}^{\left\lceil {p/2} \right\rceil - 1}{\frac{{2^{2j} }}{k^{p - 2j} }}  
+ 2^{p - 1} \left( \frac{{(- 1)^{p - 1} }}{2k + 1} + \frac{1}{2k - 1} \right),
\end{equation*}
which, multiplying through by $z^k$, gives
\begin{equation*}
\frac{z^k}{(2k - 1)k^p (2k + 1)} = - \frac{z^k}{k^p} + \sum_{j = 0}^{\left\lceil {p/2} \right\rceil - 1} \frac{{2^{2j} z^k }}{k^{p - 2j}} 
 + 2^{p - 1} z^k \left( \frac{{(- 1)^{p - 1}}}{2k + 1} + \frac{1}{2k - 1} \right).
\end{equation*}
Hence we obtain
\begin{equation}\label{eq.ch3xzuw}
\sum_{k = 1}^\infty \frac{z^k}{(2k - 1)k^p (2k + 1)}  = - \sum_{k = 1}^\infty \frac{z^k}{k^p}  
+ \sum_{j = 0}^{\left\lceil {p/2} \right\rceil - 1} {2^{2j} \sum_{k = 1}^\infty \frac{z^k}{k^{p - 2j}}}  
+ 2^{p - 1} \sum_{k = 1}^\infty z^k \left( \frac{(- 1)^{p - 1}}{2k + 1} + \frac{1}{2k - 1} \right).
\end{equation}
To evaluate the last term in~\eqref{eq.ch3xzuw}, by \eqref{macl}, let
\begin{equation*}
f(z) = \sum_{k = 1}^\infty \frac{z^k}{2k - 1} = \sqrt{z} \arctanh (\sqrt{z}) = \frac{\sqrt z }{2} \ln \left( \frac{{1 + \sqrt z }}{{1 - \sqrt z }} \right), \quad|z|<1.
\end{equation*}
Then
\begin{align}
\sum_{k = 1}^\infty z^k \left( \frac{{( - 1)^{p - 1} }}{2k + 1} + \frac{1}{2k - 1} \right)
&= (- 1)^{p - 1} \sum_{k = 1}^\infty \frac{z^k}{2k + 1} + f(z) \nonumber \\
&= (- 1)^{p - 1} \sum_{k = 2}^\infty \frac{z^{k - 1}}{2k - 1} + f(z) \nonumber \\
&= \frac{(- 1)^{p - 1}}{z} \left( \sum_{k = 1}^\infty \frac{z^k}{2k - 1} - z \right) + f(z) \nonumber \\
&= \frac{(- 1)^{p - 1}}{z} f(z) + (- 1)^p + f(z) \nonumber.
\end{align}
Thus,
\begin{equation}\label{eq.kb4c73v}
\sum_{k = 1}^\infty z^k \left( \frac{{(- 1)^{p - 1} }}{2k + 1} + \frac{1}{2k - 1} \right)
= \left( 1 + \frac{{( - 1)^{p - 1} }}{2} \right ) \frac{\sqrt{z}}{2}\ln \left( \frac{1 + \sqrt{z}}{{1 - \sqrt{z} }} \right) + (- 1)^p.
\end{equation}
Using~\eqref{eq.kb4c73v} in~\eqref{eq.ch3xzuw} and invoking the definition of the polylogarithm to write the first 
two terms of~\eqref{eq.ch3xzuw} gives for $p$ a non-negative integer and $|z|<1$ the identity
\begin{equation}\label{eq.bm10lda}
\begin{split}
\sum_{k=1}^\infty \frac{z^k}{(2k - 1)k^p (2k + 1)} &= - \Li_p (z) - \sum_{j = 1}^{\left\lceil {p/2} \right\rceil - 1} 2^{2j}\Li_{p - 2j} (z) \\ 
&\qquad + 2^{p-2} \left(1 + \frac{(- 1)^{p - 1}}{z} \right)\sqrt{z}\,\ln \left( \frac{1 + \sqrt{z}}{1 - \sqrt{z}} \right) + 2^{p - 1} (- 1)^p.
\end{split}
\end{equation}
This is an equivalent form of Theorem \ref{thm_1}.

\section{New two-term dilogarithm identities and the evaluations of some related series}

It is obvious that from Theorem \ref{thm_1} for suitable choices of the parameter $z$ additional interesting series
can be evaluated in closed form. In this section, we focus on series involving the golden section $\alpha$ and Lucas numbers $L_n$. \\

Recall that the Fibonacci numbers $F_n$ and the Lucas numbers $L_n$ are defined, for \text{$n\in\mathbb Z$}, 
through the recurrence relations $F_n = F_{n-1}+F_{n-2}$, $n\ge 2$, with initial values $F_0 = 0$, $F_1 = 1$ 
and $L_n = L_{n-1}+L_{n-2}$ with $L_0 = 2$, $L_1 = 1$. For negative subscripts we have $F_{-n} = (-1)^{n-1}F_n$ 
and $L_{-n}=(-1)^n L_n$. They possess the explicit formulas (Binet forms)
\begin{equation}\label{binet}
F_n = \frac{\alpha^n - \beta^n }{\alpha - \beta },\quad L_n = \alpha^n + \beta^n,\quad n\in\mathbb Z,
\end{equation}
where
\begin{equation*}
\alpha = \frac{1+\sqrt{5}}{2} \qquad \mbox{and} \qquad \beta = - \frac{1}{\alpha} = \frac{1-\sqrt{5}}{2}.
\end{equation*}
The number $\alpha$ is the famous golden section. The relation $L_n=F_{2n}/F_{n}$ is obvious and follows directly from \eqref{binet}.
Excellent references concerning Fibonacci and Lucas numbers are the books by Koshy \cite{Koshy} and Vajda \cite{Vajda}. 

At this point we can extend the identities in Corollary~\ref{cor.duqhfnh} to Fibonacci and Lucas numbers.

\begin{theorem}
If $n$ is a positive even integer, then
\begin{equation}
\sum_{k = 1}^\infty \frac{{5^k F_n^{2k} }}{{L_n^{2k} (2k - 1)(2k)(2k + 1)}} 
= \ln (2) - \ln (L_n) + \frac{n}{\sqrt{5}} \frac{L_{2n}}{F_{2n}}\ln (\alpha) - \frac{1}{2}.
\end{equation}
\end{theorem}
\begin{proof}
Set $z=F_n\sqrt 5/L_n$ in~\eqref{eq.ajavxbw}. When simplifying use the Binet formulas and the basic identities
\begin{equation*}
L_n^2 = 5F_n^2 + (-1)^n 4, \qquad 5 F_n^2 = L_{2n} + (-1)^{n+1} 2, \qquad L_n^2 = L_{2n} + (-1)^n 2.
\end{equation*}
\end{proof}
\begin{theorem}
If $n$ is an odd integer, then
\begin{equation}
\sum_{k = 1}^\infty \frac{L_n^{2k}}{F_n^{2k} 5^k (2k - 1)(2k)(2k + 1)} = \ln (2) - \frac{1}{2} \ln (5) 
- \ln (F_n) + \frac{n}{{\sqrt 5 }} \frac{L_{2n}}{F_{2n}}\ln (\alpha) - \frac{1}{2}.
\end{equation}
\end{theorem}
\begin{proof}
Set $z=L_n/F_n\sqrt 5$ in~\eqref{eq.ajavxbw}.
\end{proof}

\begin{theorem}
If $r$ is an even integer and $s$ is any integer, then
\begin{equation}\label{eq.se119f7}
\begin{split}
&\sum_{k = 1}^\infty  {\frac{{L_{2rk + s} }}{{L_r^{2k} (2k - 1)(2k)(2k + 1)}}} \\
&\qquad= \frac{1}{4}F_s \sqrt 5 \ln \left( {\frac{{\beta ^r L_r  + 1}}{{\alpha ^r L_r  + 1}}} \right) + \frac{1}{4}L_s \ln \left( {\frac{{2L_r^2  + 1}}{{L_r^4 }}} \right) - \frac{{L_s }}{2}\\
&\qquad\quad\; + \frac{1}{8}L_r^2 F_s \sqrt 5 \ln \left( {\left( {\frac{{L_r  + \alpha ^r }}{{L_r  + \beta ^r }}} \right)\alpha ^{2r} } \right) + \frac{1}{8}L_r^2 L_s \ln (2L_r^2  + 1)\\
&\qquad\quad\;\; - \frac{{(L_r^2  - 1)}}{{8L_r }}F_{r + s} \sqrt 5 \ln \left( {\left( {\frac{{L_r  + \alpha ^r }}{{L_r  + \beta ^r }}} \right)\alpha ^{2r} } \right) - \frac{{(L_r^2  - 1)}}{{8L_r }}L_{r + s} \ln (2L_r^2  + 1),
\end{split}
\end{equation}
\begin{equation}\label{eq.jthv04c}
\begin{split}
&\sum_{k = 1}^\infty  {\frac{{F_{2rk + s} }}{{L_r^{2k} (2k - 1)(2k)(2k + 1)}}} \\
&\qquad= \frac{1}{4}\frac{L_s}{\sqrt 5} \ln \left( {\frac{{\beta ^r L_r  + 1}}{{\alpha ^r L_r  + 1}}} \right) + \frac{1}{4}F_s  \ln \left( {\frac{{2L_r^2  + 1}}{{L_r^4 }}} \right) - \frac{{F_s }}{2}\\
&\qquad\quad\; + \frac{1}{8}L_r^2 \frac{L_s}{\sqrt 5} \ln \left( {\left( {\frac{{L_r  + \alpha ^r }}{{L_r  + \beta ^r }}} \right)\alpha ^{2r} } \right) + \frac{1}{8}L_r^2 F_s  \ln (2L_r^2  + 1)\\
&\qquad\quad\;\; - \frac{{(L_r^2  - 1)}}{{8L_r }}\frac{L_{r + s}}{\sqrt 5} \ln \left( {\left( {\frac{{L_r  + \alpha ^r }}{{L_r  + \beta ^r }}} \right)\alpha ^{2r} } \right) - \frac{{(L_r^2  - 1)}}{{8L_r }}F_{r + s} \ln (2L_r^2  + 1).
\end{split}
\end{equation}
\end{theorem}
\begin{proof}
Set $z=\alpha^r/L_r$ and $z=\beta^r/L_r$, in turn, in~\eqref{eq.ajavxbw}; add and subtract the resulting identities to obtain~\eqref{eq.se119f7} and~\eqref{eq.jthv04c}.
\end{proof}

\begin{theorem}
If $s$ is any integer, then
\begin{gather}
\sum_{k = 1}^\infty \frac{{L_{2k + s} }}{{4^k (2k - 1)(2k)(2k + 1)}} =  - L_s \ln(2) + \frac{5}{{16}}L_{s + 1} \ln(5) 
+ \frac{{\sqrt 5 }}{8}(15F_{s - 1}  - F_s )\ln (\alpha) - \frac{{L_s }}{2}, \label{eq.pmnct3y}\\
\sum_{k = 1}^\infty \frac{{F_{2k + s} }}{{4^k (2k - 1)(2k)(2k + 1)}} =  - F_s \ln(2) + \frac{5}{{16}}F_{s + 1} \ln(5) 
+ \frac{1}{{8\sqrt 5 }}(15L_{s - 1}  - L_s )\ln (\alpha) - \frac{{F_s }}{2} \label{eq.i9wdbea}.
\end{gather}
\end{theorem}
\begin{proof}
Set $z=\alpha/2$ and $z=\beta/2$, in turn, in~\eqref{eq.ajavxbw}; add and subtract the resulting identities to obtain~\eqref{eq.pmnct3y} and~\eqref{eq.i9wdbea}.
\end{proof}

To prove the main results of this section we will need the following nontrivial identities for the dilogarithm \cite{Lewin}:
\begin{equation}\label{dilog_1}
\Li_2 (-1) = - \frac{\pi^2}{12},
\end{equation}
\begin{equation}\label{dilog_2}
\Li_2 (- \alpha) = - \frac{\pi^2}{10} - \ln^2 (\alpha),
\end{equation}
\begin{equation}\label{dilog_3}
\Li_2 (- \beta) = \Li_2\Big (\frac{1}{\alpha} \Big ) = \frac{\pi^2}{10} - \ln^2 (\alpha),
\end{equation}
\begin{equation}\label{dilog_4}
\Li_2 (\beta^2 ) = \Li_2\Big (\frac{1}{\alpha^2} \Big ) = \frac{\pi^2}{15} - \ln^2 (\alpha).
\end{equation}

Inserting $z=1/\alpha$ and $z=1/\sqrt{\alpha}$ in \eqref{eq.ajavxbw} and \eqref{eq.ekwvw5d}, respectively, 
using trivial properties of the golden section, and \eqref{dilog_3} and \eqref{dilog_4} we get the evaluations
\begin{equation*}
\sum_{k=1}^\infty \frac{1}{\alpha^{2k} (2k-1)(2k+1)} = \frac{1}{2} - \frac{3}{4} \ln (\alpha),
\end{equation*}
\begin{equation*}
\sum_{k=1}^\infty \frac{1}{\alpha^{2k} (2k-1)(2k)(2k+1)} = \frac{1}{2}\Big ( \frac{3\sqrt{5}}{2} \ln (\alpha) - \ln (\alpha) - 1\Big ),
\end{equation*}
\begin{equation*}
\sum_{k=1}^\infty \frac{1}{\alpha^{2k} (2k-1)(2k)^2 (2k+1)} 
= \frac{1}{2} - \frac{\pi^2}{60} - \frac{3}{4} \ln (\alpha) + \frac{1}{4} \ln^2 (\alpha),
\end{equation*}
as well as
\begin{equation*}
\sum_{k=1}^\infty \frac{1}{\alpha^{k} (2k-1)(2k+1)} 
= \frac{1}{2} - \frac{1}{2 \alpha \sqrt{\alpha}} \arctanh \Big (\frac{1}{\sqrt{\alpha}} \Big ),
\end{equation*}
\begin{equation*}
\sum_{k=1}^\infty \frac{1}{\alpha^{k} (2k-1)(2k)(2k+1)} 
= \frac{ \alpha \sqrt{\alpha}}{2} \arctanh \Big (\frac{1}{\sqrt{\alpha}} \Big ) - \frac{1}{2} - \ln (\alpha ),
\end{equation*}
\begin{equation*}
\sum_{k=1}^\infty \frac{1}{\alpha^{k} (2k-1)(2k)^2 (2k+1)} 
= \frac{1}{2} - \frac{\pi^2}{40} - \frac{1}{2 \alpha \sqrt{\alpha}} \arctanh \Big (\frac{1}{\sqrt{\alpha}} \Big ) + \frac{1}{4} \ln^2 (\alpha).
\end{equation*} 

A series involving $k^3$ in the denominator comes from combining \eqref{eq.uu2c0rv} with the identity (consult \cite{Lewin} for a derivation)
\begin{equation}\label{trilog_id}
\Li_3 \Big ( \frac{1}{\alpha^2} \Big ) = \frac{4}{5} \zeta (3) + \frac{2}{3} \ln^3 (\alpha) - \frac{2}{15} \pi^2 \ln (\alpha).
\end{equation}
The result is
\begin{equation*}
\sum_{k=1}^\infty \frac{1}{\alpha^{2k} (2k-1)(2k)^3 (2k+1)} 
= \Big ( \frac{\pi^2}{60} + 1 + \frac{3}{4\alpha^3}\Big ) \ln (\alpha) - \frac{1}{12} \ln^3 (\alpha) - \frac{1}{10} \zeta (3) - \frac{1}{2}.
\end{equation*} \\

Recently, Campbell published two papers about dilogarithm identities \cite{Camp1,Camp2}. 
The paper \cite{Camp1} is about extending the work of Lima \cite{Lima} via Fourier–Legendre theory (polynomial expansion).
It contains five two-term dilogarithm identities which are rediscoveries of previously known dilogarithm identities.
All five results are well-documented in the book by Lewin \cite{Lewin}. One such identity is \cite[Eq. (9)]{Camp1} 
\begin{equation}
\Li_2\left (\frac{1}{\alpha^3} \right ) - \Li_2 (\beta^3 ) = \frac{\alpha^3(\pi^2 - 18 \ln^2 (\alpha) )}{3(\alpha^6 - 1)},
\end{equation}
which in view of
\begin{equation*}
\alpha^6 - 1 = (\alpha^3 - 1)(\alpha^3 + 1) = 4 \alpha^3,
\end{equation*}
can be nicely simplified resulting in
\begin{equation}\label{Campbell1}
\Li_2\left (\frac{1}{\alpha^3} \right ) - \Li_2 (\beta^3 ) = \frac{\pi^2}{12} - \frac{3}{2} \ln^2 (\alpha).
\end{equation}

Identity \eqref{Campbell1} is identity (1.70) of Lewin. Also Lima's main result from 2012 \cite{Lima} 
is Lewin's equation (1.68). Campbell's paper \cite{Camp2} is an addendum to his first 
publication \cite{Camp1}, where he references Lewin's work, discusses his results from \cite{Camp1}, and gives a historical survey. \\

We also recommend the papers by Boyadzhiev and Manns \cite{Boyadzhiev} and Stewart \cite{Stewart}. 
Boaydzhiev and Manns discuss several topics related to polylogarithms with special focus on dilogarithms. 
Stewart offers a number of proofs for ``Lima's identity'' while making use of known functional relations 
for the dilogarithm function. \\

In what follows, we present presumably new nontrivial two-term dilogarithm identities involving the golden section 
based on Lewin's book \cite{Lewin}. Such relations can be derived in a fairly straightforward manner using certain transformations.

\begin{theorem}\label{thm_2}
We have the following relations: 
\begin{equation}\label{main_21}
\Li_2\left (\frac{\alpha}{2} \right ) + \Li_2\left (\frac{\beta}{2} \right ) = \frac{\pi^2}{12} + 2\ln^2 (\alpha) - \ln^2 (2),
\end{equation}
\begin{equation}\label{main_22}
\Li_2\left (\frac{\alpha^3}{5} \right ) + \Li_2\left (\frac{\beta^3}{5} \right ) = \frac{\pi^2}{12} + 6\ln^2 (\alpha) - 2 \ln^2 (2)
+ 2\ln(2) \ln(5) - \ln^2 (5) - \Li_2 \Big(- \frac{1}{4} \Big ),
\end{equation}
\begin{equation}\label{main_23}
\Li_2 \left( \frac{{\alpha ^r }}{{L_r }} \right) + \Li_2 \left( \frac{{\beta ^r }}{{L_r }} \right) 
= \frac{\pi^2}{6} + r^2 \ln^2 (\alpha) - \ln^2 (L_r), \qquad \mbox{$r\geq 0$, $r$ even},
\end{equation}
in particular
\begin{equation}\label{main_23_1}
\Li_2 \left( \frac{{\alpha^2}}{{3}} \right) + \Li_2 \left( \frac{{\beta ^2 }}{3} \right) 
= \frac{\pi^2}{6} + 4 \ln^2 (\alpha) - \ln^2 (3),
\end{equation}
\begin{equation}\label{main_24}
\Li_2 \left( \frac{{\alpha ^r }}{{\sqrt{5} F_r}} \right) + \Li_2 \left( \frac{{- \beta ^r }}{{\sqrt{5} F_r}} \right) 
= \frac{\pi^2}{6} + r^2 \ln^2 (\alpha) - \frac{1}{4} \ln^2 (5) - \ln(5)\ln(F_r) - \ln^2 (F_r),\quad\mbox{$r\geq 1$, $r$ odd},
\end{equation}
in particular
\begin{equation}\label{main_24_1}
\Li_2\left (\frac{\alpha}{\sqrt{5}} \right ) + \Li_2\left (\frac{-\beta}{\sqrt{5}} \right ) 
= \frac{\pi^2}{6} + \ln^2 (\alpha) - \frac{1}{4} \ln^2 (5),
\end{equation}
\begin{equation}\label{main_25}
\Li_2\left (\frac{\alpha^2}{4} \right ) + \Li_2\left (\frac{\beta^2}{4} \right ) 
= \frac{\pi^2}{6} + 2 \ln^2 (\alpha) - \frac{1}{2} \ln^2 (5) + 2\ln(2)\ln(5) - 4 \ln^2(2) - \Li_2 \Big( \frac{1}{5} \Big )
\end{equation}
and
\begin{equation}\label{main_26}
\Li_2\left (\frac{\alpha}{3} \right ) + \Li_2\left (\frac{\beta}{3} \right ) 
= \ln^2 (\alpha) - \frac{1}{4} \ln^2 (5) + \ln(3)\ln(5) - \ln^2(3) + \frac{3}{2} \Li_2 \Big( \frac{1}{5} \Big ) 
- \frac{1}{2}\Li_2 \Big( \frac{1}{25} \Big ). 
\end{equation}
\end{theorem}
\begin{proof}
Many relations of this kind follow from the two-term identity \cite{Lewin}
\begin{equation}\label{dilog_id}
\Li_2 \Big ( \frac{x}{1-x} \frac{y}{1-y} \Big ) = \Li_2 \Big ( \frac{x}{1-y} \Big ) + \Li_2 \Big ( \frac{y}{1-x} \Big )
- \Li_2 (x) - \Li_2 (y) - \ln(1-x) \ln(1-y).
\end{equation}
To prove \eqref{main_21} set $x=\alpha/2$ and $y=\beta/2$ in \eqref{dilog_id}, respectively. We have
\begin{equation*}
\frac{x}{1-x} \frac{y}{1-y} = -1, \quad \frac{x}{1-y} = -\beta, \quad \frac{y}{1-x} = -\alpha.
\end{equation*}
Hence, we get
\begin{equation*}
\Li_2\left (\frac{\alpha}{2} \right ) + \Li_2\left (\frac{\beta}{2} \right ) = \Li_2 (-\beta) + \Li_2(-\alpha) - \Li_2(-1) 
- \ln(1/(2\alpha^2))\ln(\alpha^2/2).
\end{equation*}
Equation \eqref{main_21} follows upon using \eqref{dilog_1}-\eqref{dilog_3} and simplifying. 
For \eqref{main_22} insert $x=\alpha^3/5$ and $y=\beta^3/5$ in \eqref{dilog_id}, respectively. This gives
\begin{equation*}
\Li_2\left (\frac{\alpha^3}{5} \right ) + \Li_2\left (\frac{\beta^3}{5} \right ) = 
\Li_2\left (\frac{\alpha}{2} \right ) + \Li_2\left (\frac{\beta}{2} \right ) - \Li_2 \Big(- \frac{1}{4} \Big ) 
- \ln(1-\alpha^3/5)\ln(1-\beta^3/5).
\end{equation*}
Now, use \eqref{main_21} and simplify. \\
To prove \eqref{main_23} we use the dilogarithm reflection formula
\begin{equation}\label{dilog_id2}
\Li_2 (x) + \Li_2 (1 - x) = \frac{{\pi ^2 }}{6} - \ln (x)\ln (1 - x),
\end{equation}
with $x=\alpha^r/L_r$, $r$ even, and simplify. Identity \eqref{main_23_1} is the case $r=2$ in \eqref{main_23}. \\
Identity \eqref{main_24} follows from the reflection formula \eqref{dilog_id2} with $x=\alpha^r/\sqrt{5}F_r$, $r$ odd,
after some steps of simplifications. Identity \eqref{main_24_1} is the case $r=1$ in \eqref{main_24}. \\
Next, insert $x=\alpha^2/4$ and $y=\beta^2/4$ in \eqref{dilog_id} and calculate
\begin{equation*}
\frac{x}{1-x}\frac{y}{1-y} = \frac{1}{5}, \quad \frac{x}{1-y} = \frac{\alpha}{\sqrt{5}}, \quad \frac{y}{1-x} = - \frac{\beta}{\sqrt{5}}.
\end{equation*}
Hence,
\begin{equation*}
\Li_2\left (\frac{\alpha^2}{4} \right ) + \Li_2\left (\frac{\beta^2}{4} \right ) =
\Li_2\left (\frac{\alpha}{\sqrt{5}} \right ) + \Li_2\left (\frac{-\beta}{\sqrt{5}} \right ) - \Li_2 \Big( \frac{1}{5} \Big ) 
- \ln(\sqrt{5}/(4\alpha))\ln(\sqrt{5}\alpha/4).
\end{equation*}
Using identity \eqref{main_24_1} the identity is proved after some steps of simplifications. 
Finally, the choices $x=\alpha/3$ and $y=\beta/3$ in \eqref{dilog_id} yield
\begin{equation*}
\Li_2\left (\frac{\alpha}{3} \right ) + \Li_2\left (\frac{\beta}{3} \right ) = \Li_2\left (\frac{1}{\sqrt{5}} \right ) 
+ \Li_2\left (-\frac{1}{\sqrt{5}} \right ) - \Li_2\left (-\frac{1}{5} \right ) - \ln (1-\alpha/3)\ln(1-\beta/3).
\end{equation*}
Form here apply the dilogarithm identity
\begin{equation*}
\Li_2 (x) + \Li_2 (-x) = \frac{1}{2} \Li_2 (x^2) 
\end{equation*}
twice and simplify.
\end{proof}

\begin{remark}
We observe that \eqref{main_24_1} can also be proved using \eqref{dilog_id} with $x=\alpha/\sqrt{5}$ 
and $y=-\beta/\sqrt{5}$, respectively. Since
\begin{equation*}
\frac{x}{1-x} \frac{y}{1-y} = \frac{x}{1-y} = \frac{y}{1-x} = 1,
\end{equation*}
we get the striking simple relation
\begin{equation*}
\Li_2\left (\frac{\alpha}{\sqrt{5}} \right ) + \Li_2\left (\frac{-\beta}{\sqrt{5}} \right ) 
= \Li_2(1) + \ln(\alpha/\sqrt{5})\ln(\sqrt{5}\alpha).
\end{equation*}
As $\Li_2(1)=\zeta(2)$ the proof is completed. 
\end{remark}

It is worth noting that each of the equations \eqref{main_21}-\eqref{main_26} can be stated equivalently 
as an infinite sum involving Lucas and Fibonacci numbers:
\begin{equation}
\sum_{k=1}^\infty \frac{L_k}{2^k k^2} = \frac{\pi^2}{12} + 2\ln^2 (\alpha) - \ln^2 (2),
\end{equation}
\begin{equation}
\sum_{k=1}^\infty \frac{L_{3k}}{5^k k^2} = \frac{\pi^2}{12} + 6\ln^2 (\alpha) - 2 \ln^2 (2)
+ 2\ln(2) \ln(5) - \ln^2 (5) - \sum_{k=1}^\infty \frac{(-1)^k}{4^k k^2},
\end{equation}
\begin{equation}\label{special_Luc}
\sum_{k = 1}^\infty \frac{{L_{rk}}}{{L_r^k k^2 }} = \frac{\pi^2}{6} + r^2 \ln^2 (\alpha) - \ln^2 (L_r), \qquad\mbox{$r\geq 0$, $r$ even},
\end{equation}
\begin{eqnarray}
&& \sum_{k = 1}^\infty \frac{{L_{2rk}}}{{F_r^{2k} 5^k (2k)^2 }} + \sum_{k = 1}^\infty \frac{F_{r(2k - 1)}}{F_r^{2k - 1} 5^{k - 1} (2k - 1)^2}  
= \frac{\pi^2}{6} + r^2 \ln^2 (\alpha) \nonumber \\
& & \qquad\qquad - \frac{1}{4} \ln^2 (5) - \ln(5)\ln(F_r) - \ln^2 (F_r),\quad\mbox{$r\geq 1$, $r$ odd},
\end{eqnarray}
\begin{equation}
\sum_{k=1}^\infty \frac{L_{2k}}{4^k k^2} = \frac{\pi^2}{6} + 2\ln^2 (\alpha) - \frac{1}{2} \ln^2 (5) + 2 \ln(2)\ln(5) - 4 \ln^2 (2)
- \sum_{k=1}^\infty \frac{1}{5^k k^2}
\end{equation}
and
\begin{equation}
\sum_{k=1}^\infty \frac{L_{k}}{3^k k^2} = \ln^2 (\alpha) - \frac{1}{4} \ln^2 (5) + \ln(3)\ln(5) - \ln^2 (3)
+ \frac{3}{2} \sum_{k=1}^\infty \frac{1}{5^k k^2} - \frac{1}{2} \sum_{k=1}^\infty \frac{1}{5^{2k} k^2}.
\end{equation}

\begin{theorem}\label{thm_3}
The following series involving Lucas numbers $L_n$ allow a closed form evaluation 
\begin{eqnarray}\label{main_31}
\sum_{k=1}^\infty \frac{L_k}{2^{k+2} (2k-1) k^2 (2k+1)} 
& = & 1 - \frac{\pi^2}{48} + \frac{1}{4} \ln^2 (2) - \frac{1}{2} \ln^2 (\alpha)  \nonumber \\
& & - \frac{1}{2 \alpha^2 \sqrt{2 \alpha}} \arctanh \Big (\sqrt{\frac{\alpha}{2}} \Big ) 
- \frac{\alpha^2 \sqrt{\alpha}}{2 \sqrt{2}} \arctan \Big (\sqrt{\frac{1}{2 \alpha}} \Big ),
\end{eqnarray}
\begin{eqnarray}\label{main_32}
&& \sum_{k=1}^\infty \frac{L_{3k}}{5^{k} (2k-1) (2k)^2 (2k+1)} 
= 1 - \frac{\pi^2}{48} + \frac{1}{2} \ln^2 (2) - \frac{3}{2} \ln^2 (\alpha) - \frac{1}{2} \ln(2)\ln(5) + \frac{1}{4} \ln^2(5) \nonumber \\
& & \qquad + \frac{1}{4} \Li_2 \Big (- \frac{1}{4} \Big )
- \frac{1}{\alpha^3 \sqrt{5 \alpha}} \arctanh \Big ( \frac{\alpha \sqrt{\alpha}}{\sqrt{5}} \Big ) 
- \frac{\alpha^3 \sqrt{\alpha}}{\sqrt{5}} \arctan \Big ( \frac{1}{\alpha \sqrt{5 \alpha}} \Big ),
\end{eqnarray}
\begin{equation}\label{main_33}
\begin{split}
\sum_{k = 1}^\infty \frac{{L_{rk} }}{{L_r^k (2k - 1)(2k)^2 (2k + 1)}}  &= 1 - \frac{\pi ^2}{24} - \frac{1}{4} r^2 \ln^2 (\alpha) 
+ \frac{1}{4} \ln^2 (L_r) - \frac{1}{2} \sqrt {\frac{1}{\alpha^{3r} L_r }} \arctanh\left( {\sqrt {\frac{{\alpha ^r }}{{L_r }}} } \right)\\
&\qquad - \frac{1}{2} \sqrt {\frac{\alpha ^{3r}}{L_r }} \arctanh\left( \sqrt {\frac{1}{{\alpha^r L_r }}} \right),
\quad\mbox{$r\in\mathbb N_0$, $r$ even},
\end{split}
\end{equation}
\begin{eqnarray}\label{main_34}
&& \sum_{k=1}^\infty \frac{L_{2k}}{5^{k} (4k-1) (4k)^2 (4k+1)} + \sum_{k=1}^\infty \frac{F_{2k-1}}{5^{k-1} (4k-3) (4k-2)^2 (4k-1)} 
= 1 - \frac{\pi^2}{24} + \frac{1}{16} \ln^2 (5)  \nonumber \\
&& \qquad - \frac{1}{4} \ln^2 (\alpha) - \frac{1}{2\sqrt{\sqrt{5}}}
\Big ( \frac{1}{\alpha \sqrt{\alpha}} \arctanh \Big (\sqrt{\frac{\alpha}{\sqrt{5}}} \Big ) 
+ \alpha \sqrt{\alpha} \arctanh \Big (\sqrt{\frac{1}{\sqrt{5} \alpha}} \Big ) \Big ),
\end{eqnarray}
and
\begin{eqnarray}\label{main_35}
\sum_{k=1}^\infty \frac{L_{2k}}{4^{k+1} (2k-1) k^2 (2k+1)} & = & 1 - \frac{\pi^2}{24} - \frac{1}{2} \ln^2 (\alpha) 
- \frac{9\sqrt{5}}{8}\ln (\alpha) + \frac{5}{16}\ln (5) \nonumber \\
& & + \frac{1}{8} \ln^2 (5) - \frac{1}{2}\ln (2) \ln(5) + \ln^2 (2) + \frac{1}{4} \Li_2\Big (\frac{1}{5}\Big ).
\end{eqnarray}
\end{theorem}
\begin{proof}
To prove \eqref{main_31} set $z=\sqrt{\alpha/2}$ and $z=i \sqrt{-\beta/2}$ in \eqref{main_1}, 
respectively, combine according to the Binet form, simplify and make use of \eqref{main_21}. 
To prove \eqref{main_32} set $z=\sqrt{\alpha^3/5}$ and $z=i \sqrt{-\beta^3/5}$ in \eqref{main_1}, 
respectively, combine according to the Binet form, simplify and make use of \eqref{main_22}. 
Identity \eqref{main_33} comes from setting $z=\sqrt{\alpha^r/L_r}$ and $z=\sqrt{\beta^r/L_r}$, in turn, 
in \eqref{eq.ekwvw5d} and making use of \eqref{main_23}.
To prove \eqref{main_34} set $z=\sqrt{\alpha/\sqrt{5}}$ and $z=\sqrt{-\beta/\sqrt{5}}$ in \eqref{main_1}, respectively, 
combine according to the Binet form, simplify and make use of \eqref{main_24_1}. 
Finally, proceed as before with $z=\alpha/2$ and $z=\beta/2$ in \eqref{main_1}, respectively. 
When simplifying apply the relation \eqref{main_25}. 
\end{proof}

Using the identities of Hoggatt et al. (Lemma~\ref{lem.ydalnfx}) we will extend the identity~\eqref{main_23}.

\begin{lemma}[Hoggatt et al.~\cite{hoggatt71}]\label{lem.ydalnfx}
For $p$ and $q$ integers,
\begin{gather}
F_{p + q} - F_p \alpha ^q = \beta ^p F_q, \\ 
F_{p + q} - F_p \beta ^q = \alpha ^p F_q, \\
L_{p + q} - L_p \alpha ^q = - \beta ^p F_q \sqrt 5, \\ 
L_{p + q} - L_p \beta ^q = \alpha ^p F_q \sqrt 5. 
\end{gather}
\end{lemma}
\begin{lemma}
For integers $p$ and $q$,
\begin{gather}
\Li_2 \left( {\frac{{F_p }}{{F_{p + q} }}\alpha ^q } \right) + \Li_2 \left( {\frac{{F_q }}{{F_{p + q} }}\beta ^p } \right) 
= \frac{{\pi ^2 }}{6} - \ln \left( {\frac{{F_p }}{{F_{p + q} }}\alpha ^q } \right)\ln \left( {\frac{{F_q }}{{F_{p + q} }}\beta ^p } \right),\quad p + q\ne 0,\label{eq.mddhecj}\\
\Li_2 \left( {\frac{{F_q }}{{F_{p + q} }}\alpha ^p } \right) + \Li_2 \left( {\frac{{F_p }}{{F_{p + q} }}\beta ^q } \right) 
= \frac{{\pi ^2 }}{6} - \ln \left( {\frac{{F_q }}{{F_{p + q} }}\alpha ^p } \right)\ln \left( {\frac{{F_p }}{{F_{p + q} }}\beta ^q } \right),\quad p + q\ne 0,\\
\Li_2 \left( {\frac{{L_p }}{{L_{p + q} }}\alpha ^q } \right) + \Li_2 \left( {\frac{{ - F_q \sqrt 5 }}{{L_{p + q} }}\beta ^p } \right) 
= \frac{{\pi ^2 }}{6} - \ln \left( {\frac{{L_p }}{{L_{p + q} }}\alpha ^q } \right)\ln \left( {\frac{{ - F_q \sqrt 5 }}{{L_{p + q} }}\beta ^p } \right),\\
\Li_2 \left( {\frac{{L_p }}{{L_{p + q} }}\beta ^q } \right) + \Li_2 \left( {\frac{{F_q \sqrt 5 }}{{L_{p + q} }}\alpha ^p } \right) 
= \frac{{\pi ^2 }}{6} - \ln \left( {\frac{{L_p }}{{L_{p + q} }}\beta ^q } \right)\ln \left( {\frac{{F_q \sqrt 5 }}{{L_{p + q} }}\alpha ^p } \right)\label{eq.xmhoye7}.
\end{gather}
\end{lemma}
\begin{proof}
Use the dilogarithm reflection formula \eqref{dilog_id2}.
\end{proof}

As usual, identities~\eqref{eq.mddhecj}--\eqref{eq.xmhoye7} can be stated as Fibonacci and Lucas series, 
namely, that if $p$ and $q$ are positive integers, then
\begin{equation}\label{Fib_Luc_id1}
\begin{split}
\sum_{k = 1}^\infty  {\frac{{F_p^k L_{qk}  + F_q^k L_{pk} }}{{F_{p + q}^k k^2 }}} & = \frac{{\pi ^2 }}{3} - \ln \left( {\frac{{F_p }}{{F_{p + q} }}\alpha ^q } \right)\ln \left( {\frac{{F_q }}{{F_{p + q} }}\beta ^p } \right)\\
&\qquad - \ln \left( {\frac{{F_q }}{{F_{p + q} }}\alpha ^p } \right)\ln \left( {\frac{{F_p }}{{F_{p + q} }}\beta ^q } \right),\quad\mbox{$p$ even, $q$ even},
\end{split}
\end{equation}
and
\begin{equation}\label{Fib_Luc_id2}
\begin{split}
&\sum_{k = 1}^\infty  {\frac{{L_p^k L_{kq} }}{{L_{p + q}^k k^2 }}} + \sum_{k = 1}^\infty  {\frac{{F_q^{2k} 5^k L_{2kp} }}{{L_{p + q}^{2k} (2k)^2 }}} + \sum_{k = 1}^\infty  {\frac{{F_q^{2k - 1} 5^k F_{(2k - 1)p} }}{{L_{p + q}^{2k - 1} (2k - 1)^2 }}} \\
&\qquad = \frac{{\pi ^2 }}{3} - \ln \left( {\frac{{L_p }}{{L_{p + q} }}\alpha ^q } \right)\ln \left( {\frac{{ - F_q \sqrt 5 }}{{L_{p + q} }}\beta ^p } \right) - \ln \left( {\frac{{L_p }}{{L_{p + q} }}\beta ^q } \right)\ln \left( {\frac{{F_q \sqrt 5 }}{{L_{p + q} }}\alpha ^p } \right),\quad\mbox{$p$ odd, $q$ even}.
\end{split}
\end{equation}

When $p=q=2$ then \eqref{Fib_Luc_id1} gives
\begin{equation*}
\sum_{k=1}^\infty \frac{L_{2k}}{3^k k^2} = \frac{\pi^2}{6} + 4 \ln^2 (\alpha) - \ln^2 (3),
\end{equation*}
which confirms \eqref{special_Luc} with $r=2$. When $p=1$ and $q=2$ then \eqref{Fib_Luc_id2} produces
\begin{equation*}
\begin{split}
&\sum_{k=1}^\infty \frac{L_{2k}}{4^k k^2} + \sum_{k=1}^\infty \frac{5^k L_{2k}}{4^{2k} (2k)^2} 
+ \sum_{k=1}^\infty \frac{5^k F_{2k - 1}}{4^{2k - 1} (2k - 1)^2} \\
&\qquad\qquad = \frac{{\pi^2}}{3} + 4 \ln^2 (\alpha) + 2\ln(2)\ln(5) - 8 \ln^2(2).
\end{split}
\end{equation*}

Identities \eqref{Fib_Luc_id1} and \eqref{Fib_Luc_id2} also lead to Ramanujan type sums as presented in Theorem \ref{thm_3}.
We state the result corresponding to \eqref{Fib_Luc_id1} in the next theorem and leave the other sum as an exercise.

\begin{theorem}\label{thm_4}
Let $p$ and $q$ be even integers. Then  
\begin{eqnarray}\label{Ram_61}
&& \sum_{k=1}^\infty \frac{F_p^k L_{qk}+F_q^k L_{pk}}{F_{p+q}^k (2k-1) (2k)^2 (2k+1)} = 2 - \frac{\pi^2}{12} \nonumber \\
&& \quad - \frac{1}{4} \ln \left( \frac{{F_p }}{{F_{p + q} }}\alpha ^q \right) \ln \left( \frac{{F_q }}{{F_{p + q} }}\beta ^p \right )  
- \frac{1}{4} \ln \left( \frac{{F_q }}{{F_{p + q} }}\alpha ^p \right) \ln \left( \frac{{F_p }}{{F_{p + q} }}\beta ^q \right ) \nonumber \\
&& \quad - \frac{1}{2} \frac{F_q \beta^p}{\sqrt{F_{p+q} F_p \alpha^q}} \arctanh \Big (\sqrt{\frac{F_p \alpha^q}{F_{p+q}}} \Big )
- \frac{1}{2} \frac{F_p \alpha^q}{\sqrt{F_{p+q} F_q \beta^p}} \arctanh \Big (\sqrt{\frac{F_q \beta^p}{F_{p+q}}} \Big ) \nonumber \\
&& \quad - \frac{1}{2} \frac{F_p \beta^q}{\sqrt{F_{p+q} F_q \alpha^p}} \arctanh \Big (\sqrt{\frac{F_q \alpha^p}{F_{p+q}}} \Big )
- \frac{1}{2} \frac{F_q \alpha^p}{\sqrt{F_{p+q} F_p \beta^q}} \arctanh \Big (\sqrt{\frac{F_p \beta^q}{F_{p+q}}} \Big ).
\end{eqnarray}
\end{theorem}

\section{Additional series} 

In the next couple of theorems we state identities involving binomial coefficients.

\begin{lemma}
Let $m\in\mathbb Z^+$, $m\ge 2$. Let $z$ be any real or complex variable such that $|z|<1$. Then
\begin{equation}\label{eq.qigbmpa}
\begin{split}
\sum_{k = 1}^\infty {\frac{{\binom{2k}mz^{2k} }}{{(2k - 1)(2k)(2k + 1)}}} &= \frac{1}{4z}\sum_{j = 0}^{m - 1} {\frac{1}{{m - j}}\left( {( - 1)^{m - 1} \left( {\frac{z}{{1 + z}}} \right)^{m - j}  + ( - 1)^j \left( {\frac{z}{{1 - z}}} \right)^{m - j} } \right)}\\ 
&\qquad + \frac{1}{2m}\left( {( - 1)^{m - 1} \left( {1 + \frac{z}{2}} \right)\left( {\frac{z}{{1 + z}}} \right)^m  - \left( {1 - \frac{z}{2}} \right)\left( {\frac{z}{{1 - z}}} \right)^m } \right)\\
&\qquad\; + \frac{1}{4(m - 1)}\left( {( - 1)^m (1 + z)\left( {\frac{z}{{1 + z}}} \right)^m  + (1 - z)\left( {\frac{z}{{1 - z}}} \right)^m } \right)\\
&\qquad\;\; + \frac{{( - 1)^m }}{4z}\ln \left( {\frac{{1 + z}}{{1 - z}}} \right).
\end{split}
\end{equation}
\end{lemma}
\begin{proof}
Differentiate~\eqref{eq.ajavxbw} with respect to $z$, $m$ times.
\end{proof}

\begin{theorem}
If $m$ is a positive integer greater than or equal to $2$, then
\begin{equation}\label{eq.myu60u9}
\begin{split}
\sum_{k = 1}^\infty  {\frac{\binom{2k}m}{{5^k (2k - 1)(2k)(2k + 1)}}}  &= \frac{5}{4}\sum_{j = 0}^{\left\lfloor {(m - 1)/2} \right\rfloor } {\frac{{F_{m - 2j} }}{{m - 2j}}\,\frac{1}{{2^{m - 2j} }}}  - \frac{5}{4}\sum_{j = 1}^{\left\lceil {(m - 1)/2} \right\rceil } {\frac{{F_{m - 2j + 1} }}{{m - 2j + 1}}\,\frac{1}{{2^{m - 2j + 1} }}}\\ 
&\qquad - \frac{{L_m }}{m}\,\frac{1}{{2^{m + 1} }} + \frac{{F_m }}{m}\,\frac{1}{{2^{m + 2} }} + \frac{{F_{m - 1} }}{{m - 1}}\,\frac{1}{{2^{m + 1} }} + \frac{{( - 1)^m }}{2}\sqrt 5\, \ln (\alpha) ,
\end{split}
\end{equation}
\begin{equation}\label{eq.qat8uyi}
\begin{split}
\sum_{k = 1}^\infty {\frac{{4^k \binom{2k}m}}{{5^k (2k - 1)(2k)(2k + 1)}}}  &= \frac{5}{8}\sum_{j = 0}^{\left\lfloor {(m - 1)/2} \right\rfloor } {\frac{{F_{3(m - 2j)} }}{{m - 2j}}2^{m - 2j} } - \frac{5}{8}\sum_{j = 1}^{\left\lceil {(m - 1)/2} \right\rceil } {\frac{{F_{3(m - 2j + 1)} }}{{m - 2j + 1}}2^{m - 2j + 1} }\\ 
&\qquad - \frac{{L_{3m} }}{m}2^{m - 1}  + \frac{{F_{3m} }}{m}2^{m - 1} + \frac{{F_{3(m - 1)} }}{{m - 1}}2^{m - 2} + \frac{{( - 1)^m 3}}{4}\sqrt 5 \ln (\alpha), 
\end{split}
\end{equation}
and
\begin{equation}\label{eq_Fib_bin3}
\begin{split}
\sum_{k = 1}^\infty \frac{5^k \binom {2k}{m}}{9^k (2k - 1)(2k)(2k + 1)}  &= \frac{3}{4\sqrt{5}} \sum_{j = 0}^{m - 1} \frac{A(m,j)}{m-j} \\ 
&\qquad + \frac{1}{m 2^{m+1}}
\begin{cases}
5^{m/2} \Big ( -\frac{5}{2}  L_{2m} + L_{2m-1}\Big ), & \text{$m$ even;} \\ 
5^{(m+1)/2} \Big ( -\frac{5}{2}  F_{2m} + F_{2m-1} \Big ), & \text{$m$ odd;} 
\end{cases} \\
&\qquad + \frac{1}{3(m-1) 2^{m+1}}
\begin{cases}
5^{m/2} L_{2m-2}, & \text{$m$ even;} \\ 
5^{(m+1)/2} F_{2m-2}, & \text{$m$ odd;} 
\end{cases}
\end{split}
\end{equation}
with
\begin{equation*}
A(m,j) = 
\begin{cases}
- \frac{5^{(m-j)/2}}{2^{m-j}} L_{2(m-j)}, & \text{$m$ even, $j$ odd;} \\ 
\frac{5^{(m-j)/2}}{2^{m-j}} L_{2(m-j)}, & \text{$m$ odd, $j$ even;} \\
\frac{5^{(m-j+1)/2}}{2^{m-j}} F_{2(m-j)}, & \text{$m$ even, $j$ even;} \\
- \frac{5^{(m-j+1)/2}}{2^{m-j}} F_{2(m-j)}, & \text{$m$ odd, $j$ odd.}
\end{cases}
\end{equation*}
\end{theorem}
\begin{proof}
Setting $z=1/\sqrt 5$ in \eqref{eq.qigbmpa} gives \eqref{eq.myu60u9} while $z=2/\sqrt 5$ gives \eqref{eq.qat8uyi} and 
$z=\sqrt 5/3$ gives \eqref{eq_Fib_bin3}.
\end{proof}

\begin{theorem}
If $n$ is a positive even integer and $m$ is a positive integer, then
\begin{equation}
\begin{split}
&\sum_{k = 1}^\infty  {\frac{{F_n^{2k} 5^k \binom{2k}{2m}}}{{L_n^{2k} (2k - 1)(2k)(2k + 1)}}} \\
&\qquad = \frac{{L_n }}{{4F_n }}\sum_{j = 0}^{m - 1} {\frac{{5^{m - j} }}{{2m - 2j}}\frac{{F_n^{2m - 2j} }}{{2^{2m - 2j} }}F_{n(2m - 2j)} }  - \frac{{L_n }}{{4F_n }}\sum_{j = 1}^m {\frac{{5^{m - j} }}{{2m - 2j + 1}}\frac{{F_n^{2m - 2j + 1} }}{{2^{2m - 2j + 1} }}L_{n(2m - 2j + 1)} }\\ 
&\qquad\; - \frac{{5^m }}{m}\frac{{F_n^{2m} }}{{2^{2m + 2} }}L_{2nm} + \frac{1}{{L_n }}\frac{{5^{m + 1} }}{m}\frac{{F_n^{2m + 1} }}{{2^{2m + 3} }}F_{2nm} + \frac{1}{{L_n }}\frac{{5^m }}{{(2m - 1)}}\frac{{F_n^{2m} }}{{2^{2m + 1} }}L_{n(2m - 1)}  + \frac{{nL_n }}{{2F_n \sqrt 5 }}\ln (\alpha) ,
\end{split}
\end{equation}
\begin{equation}
\begin{split}
&\sum_{k = 1}^\infty  {\frac{{F_n^{2k} 5^k \binom{2k}{2m + 1}}}{{L_n^{2k} (2k - 1)(2k)(2k + 1)}}} \\
&\qquad = \frac{{L_n }}{{4F_n }}\sum_{j = 0}^m {\frac{{5^{m - j} }}{{2m - 2j + 1}}\frac{{F_n^{2m - 2j + 1} }}{{2^{2m - 2j + 1} }}L_{n(2m - 2j + 1)} }  - \frac{{L_n }}{{4F_n }}\sum_{j = 1}^m {\frac{{5^{m - j + 1} }}{{2m - 2j + 2}}\frac{{F_n^{2m - 2j + 2} }}{{2^{2m - 2j + 2} }}F_{n(2m - 2j + 2)} } \\
&\qquad\; - \frac{{5^{m + 1} }}{{(2m + 1)}}\frac{{F_n^{2m + 1} }}{{2^{2m + 2} }}F_{n(2m + 1)}  + \frac{1}{{L_n }}\frac{{5^{m + 1} }}{{(2m + 1)}}\frac{{F_n^{2m + 2} }}{{2^{2m + 3} }}L_{n(2m + 1)} \\
&\qquad\;\; + \frac{1}{{L_n }}\frac{{5^{m + 1} }}{m}\frac{{F_n^{2m + 1} }}{{2^{2m + 3} }}F_{2mn}  - \frac{{nL_n }}{{2F_n \sqrt 5 }}\ln (\alpha).
\end{split}
\end{equation}
\end{theorem}

\begin{proof}
Set $z=F_n\sqrt 5/L_n$ in~\eqref{eq.qigbmpa}, consider the parity of $m$ and use the Binet formulas.
\end{proof}

\begin{lemma}
Let $m\in\mathbb Z^+$, $m\ge 2$. Let $z$ be any real or complex variable such that $|z|<1$. Then
\begin{equation}\label{eq.f7siyvk}
\begin{split}
\sum_{k = 1}^\infty  {\frac{{\binom{2k}mz^{2k} }}{{(2k - 1)k^2 (2k + 1)}}}  &= 2\frac{{( - 1)^{m - 1} }}{m}\ln (1 - z^2 ) + \frac{{( - 1)^m }}{z}\ln \left( {\frac{{1 - z}}{{1 + z}}} \right)\\
&\qquad + \frac{1}{{m - 1}}\left( {( - 1)^m (1 + z)\left( {\frac{z}{{1 + z}}} \right)^m  + (1 - z)\left( {\frac{z}{{1 - z}}} \right)^m } \right)\\
&\qquad\,- \frac{1}{m}\left( {( - 1)^m z\left( {\frac{z}{{1 + z}}} \right)^m  - z\left( {\frac{z}{{1 - z}}} \right)^m } \right)\\
&\qquad\;+ \frac{2}{m}\sum_{j = 1}^{m - 1} {\frac{1}{{m - j}}\left( {( - 1)^m \left( {\frac{z}{{1 + z}}} \right)^{m - j}  + ( - 1)^j \left( {\frac{z}{{1 - z}}} \right)^{m - j} } \right)} \\
&\qquad\,\;+ \sum_{j = 0}^{m - 1} {\frac{1}{{m - j}}\left( {( - 1)^m \frac{1}{z}\left( {\frac{z}{{1 + z}}} \right)^{m - j}  - ( - 1)^j \frac{1}{z}\left( {\frac{z}{{1 - z}}} \right)^{m - j} } \right)}. 
\end{split}
\end{equation}
\end{lemma}
\begin{proof}
Differentiate $4F(z,2)$ given in~\eqref{eq.ekwvw5d} $m$ times with respect to $z$.
\end{proof}

\begin{theorem}
If $m$ is a positive integer greater than or equal to $2$, then
\begin{equation}\label{eq.ikwbv9a}
\begin{split}
&\sum_{k = 1}^\infty  {\frac{\binom{2k}m}{{5^k (2k - 1)k^2 (2k + 1)}}} \\
&\qquad= 2\frac{{( - 1)^m }}{m}\ln \left( {\frac{5}{4}} \right) - ( - 1)^m 2\sqrt 5 \ln (\alpha) + \frac{{F_{m - 1} }}{{m - 1}}\frac{1}{{2^{m - 1} }} + \frac{{F_m }}{m}\frac{1}{{2^m }}\\
&\qquad\,+ \frac{2}{m}\sum_{j = 1}^{\left\lfloor {(m - 1)/2} \right\rfloor } {\frac{{L_{m - 2j} }}{{m - 2j}}\frac{1}{{2^{m - 2j} }}}  - \frac{2}{m}\sum_{j = 1}^{\left\lceil {(m - 1)/2} \right\rceil } {\frac{{L_{m - 2j + 1} }}{{m - 2j + 1}}\frac{1}{{2^{m - 2j + 1} }}} \\
&\qquad\;- 5\sum_{j = 0}^{\left\lfloor {(m - 1)/2} \right\rfloor } {\frac{{F_{m - 2j} }}{{m - 2j}}\frac{1}{{2^{m - 2j} }}}  + 5\sum_{j = 1}^{\left\lceil {(m - 1)/2} \right\rceil } {\frac{{F_{m - 2j + 1} }}{{m - 2j + 1}}\frac{1}{{2^{m - 2j + 1} }}} .
\end{split}
\end{equation}
\end{theorem}

\begin{proof}
Set $z=1/\sqrt 5$ in~\eqref{eq.f7siyvk}.
\end{proof}

\end{document}